\documentclass[12pt]{article}
\usepackage{mycommands}
\usepackage{tikz}
\usetikzlibrary{decorations.pathreplacing}
\title{Constructing Reducible Brill--Noether Curves II}

\begin{document}
\maketitle

\begin{abstract}
In this paper, we study maps from reducible curves $f \colon C \cup_\Gamma D \to \pp^r$.
We restrict our attention to two cases:
first, when $f|_D$ factors through a hyperplane $H$ and $f|_C$ is transverse to $H$;
and second, when $r = 3$.
Degeneration to stable maps of this type have played a crucial role in works of Hartshorne, Ballico, and others,
on special cases of the Maximal Rank Conjecture.

However, the general problem of studying when such stable maps
with specified combinatorial types exist remains open.
Here, we give criteria for such Brill--Noether curves of this first type
to exist, of
specified degree $d$ and genus $g$,
such that $f|_C$ is of specified degree $d'$ and genus $g'$.
We also give criteria, sharpening earlier results of \cite{rbn},
for the existence of Brill--Noether space
curves of specified combinatorial types.

As explained in \cite{over}, these results play a key role in the author's proof
of the Maximal Rank Conjecture \cite{mrc}. 
\end{abstract}

\section{Introduction}

The technique of degeneration to a reducible curve has enabled the proof of many results in
the theory of algebraic curves.
These results include
the Brill--Noether theorem and related results \cite{eis-har},
the existence of components of the Hilbert scheme with the expected number
of moduli when the Brill--Noether number is negative \cite{sernesi},
and (using the present work) the Maximal Rank Conjecture \cite{mrc}.

In this paper, we first explore a specific type of degeneration used in the work of Hirschowitz \cite{mrat},
Ballico (\cite{ball2}, \cite{ball}, etc.), and others, to study various special cases of the Maximal Rank Conjecture in $\pp^r$ for $r \geq 4$: Degeneration
of stable maps of degree $d$ and genus $g$
to stable maps from reducible curves $f \colon C \cup_\Gamma D \to \pp^r$, where
$f|_D$ factors as $\iota \circ f_D$ for $\iota \colon H \hookrightarrow \pp^r$ the inclusion
of a hyperplane $H \subset \pp^r$, and $f|_C$ is
transverse to $H$ and of specified degree $d'$ and genus $g'$, and $f(\Gamma)$ is a set of $n$ general points in $H$.
We will show such degenerations exist, subject to certain
numerical constraints, for components of Kontsevich's space of stable maps which
dominate the moduli space of curves; such degenerations may therefore
be used in the study of the geometry of general curves.
In particular, as explained in \cite{over}, they may be used
to prove the Maximal Rank Conjecture.

Second, we explore degenerations of space curves to reducible curves with general nodes,
sharpening earlier results of \cite{rbn} in the case $r = 3$.
Namely,
one method to construct stable maps from reducible curves to $\pp^3$
is to first take a finite nonempty set
$\Gamma$ of general points in $\pp^3$,
and find maps $f_i \colon C_i \to \pp^r$ from curves $C_i$ of general moduli
of specified degrees and genera
which pass through~$\Gamma$.

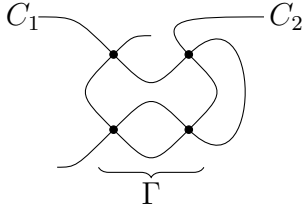
\begin{center}
\begin{tikzpicture}
\filldraw (1, 1) circle [radius=0.05];
\filldraw (2, 1) circle [radius=0.05];
\filldraw (1, 2) circle [radius=0.05];
\filldraw (2, 2) circle [radius=0.05];
\draw (0, 2.5) .. controls (0.5, 2.5) .. (1, 2);
\draw (0.25, 0.5) .. controls (0.5, 0.5) .. (1, 1);
\draw (1, 1) .. controls (1.5, 1.5) .. (2, 1);
\draw (1, 2) .. controls (1.5, 1.5) .. (2, 2);
\draw (2, 1) .. controls (3, 0) and (3, 3) .. (2, 2);
\draw (3, 2.5) .. controls (2, 2.5) and (1.5, 2.5) .. (2, 2);
\draw (2, 2) .. controls (2.5, 1.5) .. (2, 1);
\draw (2, 1) .. controls (1.5, 0.5) .. (1, 1);
\draw (1, 1) .. controls (0.5, 1.5) .. (1, 2);
\draw (1, 2) .. controls (1.25, 2.25) .. (1.5, 2.25);
\draw (-0.2, 2.5) node{$C_1$};
\draw (3.3, 2.5) node{$C_2$};
\draw [decorate, decoration={brace, mirror, amplitude=0.75ex}] (0.8, 0.5) -- (2.2, 0.5);
\draw (1.5, 0.15) node{$\Gamma$};
\end{tikzpicture}
\end{center}

\noindent
Results of \cite{vogt} determine exactly when such
curves may be found:

\begin{thm}[Theorem~1.1 of \cite{vogt}] \label{thm11}
There exists a nondegenerate map from a curve $C \to \pp^3$ of degree $d$ and genus $g$
with $C$ of general moduli,
passing through $n$ general points in $\pp^3$,
if and only if
$4d - 3g - 12 \geq 0$ and
\[\begin{cases}
n \leq 2d & \text{if $(d, g) \notin \{(5, 2), (6, 4)\}$;} \\
n \leq 9 & \text{if $(d, g) \in \{(5, 2), (6, 4)\}$.} \\
\end{cases}\]
\end{thm}

We then pick subsets $\Gamma_i \subset C_i$ which map injectively under $f_i$ onto $\Gamma$, and
write $C_1 \cup_\Gamma C_2$ for (a choice of) the curve obtained
from $C_1$ and $C_2$ by gluing 
$\Gamma_1$ to $\Gamma_2$ via the obvious isomorphism.
The maps $f_i$ give rise to a map $f \colon C_1 \cup_\Gamma C_2 \to \pp^3$.
Note that even when the $f_i$ are immersions of smooth curves,
$f$ could fail to be an immersion (for example if $C_1$ and $C_2$ meet at an additional
point not in $\Gamma$). This is the reason why we work with stable
maps, instead of taking the scheme-theoretic union of the corresponding curves in projective space.

In this paper, we will show
that the
stable maps $C_1 \cup_\Gamma C_2 \to \pp^3$ constructed in this manner
lie in the closure of the
locus of stable maps from curves of general moduli,
subject to certain mild constraints.
In particular, they may therefore
be used in the study of the geometry of stable maps from general curves.

\medskip

More precisely, write $\bar{M}_g(\pp^r, d)$ (respectively $\mathcal{H}_{d, g, r}$ assuming that $r \geq 3$)
for the Kontsevich space (respectively Hilbert scheme) which
classifies stable maps $C \to \pp^r$ of degree~$d$ from a nodal curve $C$
of genus~$g$ (respectively subschemes of $\pp^r$ with Hilbert polynomial $P(x) = dx + 1 - g$).
If $X$ is a component of $\bar{M}_g(\pp^r, d)$ (respectively $\mathcal{H}_{d, g, r}$)
whose generic member is a map
from a smooth curve, which is an immersion if $r \geq 3$,
birational onto its image if $r = 2$, and finite if $r = 1$
(respectively is a smooth curve),
there is a natural map (respectively rational map)
$X \to \bar{M}_g$.
We refer to a stable map $C \to \pp^r$
(respectively a subscheme $C \subset \pp^r$ for $r \geq 3$)
as a \emph{Brill--Noether curve} (\emph{BN-curve}) if it
corresponds to a point in some such component $X$
which both dominates $\bar{M}_g$, and whose generic member is nondegenerate.
Moreover, we say a stable map $C \to \pp^r$
(respectively a curve $C \subset \pp^r$ for $r \geq 3$)
is an \emph{interior curve} if it lies in a unique component
of the Kontsevich space (respectively Hilbert scheme).

The Brill--Noether theorem
asserts that BN-curves of degree $d$ and genus $g$ in $\pp^r$ exist
if and only if the \emph{Brill--Noether number}
\[\rho(d, g, r) := (r + 1)d - rg - r(r + 1) \geq 0;\]
and that in this case, the locus of BN-curves forms an irreducible
component of $\bar{M}_g(\pp^r, d)$.
(And, if $r \geq 3$, forms an irreducible component of $\mathcal{H}_{d,g,r}$
which is birational to the component of  $\bar{M}_g(\pp^r, d)$ corresponding to BN-curves.)

\medskip

Then our first goal is to construct reducible BN-curves $f \colon C \cup_\Gamma D \to \pp^r$
of the above form, where both $f|_C$ and $f_D$ are BN-curves too.
Since the genus of $D$ and degree of $f_D$ are
\[d'' = d - d' \tand g'' = g + 1 - g' - n,\]
and $C \cup_\Gamma D$ must be connected, and the hyperplane section $f(C) \cap H$
contains $d'$ points (or fewer),
we must have
\begin{align*}
g' &\geq 0 \\
g + 1 - g' - n &= g'' \geq 0 \\
(r + 1) d' - rg' - r^2 - r &= \rho(d', g', r) \geq 0 \\
r(d - d') - (r - 1)(g - g') + (r - 1) n - r^2 + 1 &= \rho(d'', g'', r - 1) \geq 0 \\
n - 1 &\geq 0 \\
d' - n &\geq 0.
\end{align*}

In order to construct such reducible curves $C \cup_\Gamma D \to \pp^r$,
we first need to know when we can pass $f|_C$ and $f_D$ through a set $\Gamma \subset H$ of $n$ general points.
In this paper, we will focus on the case when these are guaranteed by results of \cite{ibe}
(although our method will be quite general and would in particular apply
whenever we have inequalities of a reasonable shape that guarantee this).
Namely,
Theorem~1.5 of~\cite{ibe} implies the hyperplane section of $f|_C$
can pass through $n$ general points subject to the inequality
\begin{equation} \label{cC}
(2r - 3) (d' + 1) - (r - 2)^2 (g' - d' + n) - 2r^2 + 3r - 9 \geq 0.
\end{equation}
In addition, by Theorem~1.2 of~\cite{ibe}, $f_D$ passes through $n$ general points provided that
\[(r - 2)n \leq r d'' - (r - 4)(g'' - 1) - 2r + 2;\]
or upon rearrangement,
\[r (d - d') - (r - 4)(g - g') - 2n - 2r + 2 \geq 0.\]

When all of these inequalities are satisfied, we can construct
such a curve $C \cup_\Gamma D \to \pp^r$;
but a priori, this curve may not be a BN-curve --- in fact, a priori, it may not even lie in a component of the Kontsevich space
whose generic member is a map from a smooth curve.
One can show,
as in the proof of Corollary~4.3
of \cite{hh} mutatis mutandis,
that when $f|_C$ and $f_D$ are general,
$C \cup_\Gamma D \to \pp^r$ admits a deformation which
is a map from a smooth curve
provided that
\begin{equation} \label{nC}
2n + d + g' - d' - g - r - 1 = n + d'' - g'' - r = n - (\dim H^1 (N_{f|_D}) + 1) \geq 0.
\end{equation}

In these terms, our first theorem shows that, if there exists an $n$ satisfying these
inequalities, with \eqref{cC} satisfied even when $d'$ is decreased by $1$
and \eqref{nC} strict,
then for the minimal such $n$, the resulting curve $C \cup_\Gamma D \to \pp^r$
is in fact a BN-curve.
Namely:

\begin{thm} \label{main}
Let $d$, $g$, $d'$, $g'$, and $r$ be integers
which satisfy:
\begin{align}
g' &\geq 0 \label{b} \\
(r + 1) d - rg - r^2 - r &\geq 0 \label{c} \\
(r + 1) d' - rg' - r^2 - r &\geq 0 \label{d} \\
\intertext{Suppose there exists an integer $n$ satisfying:}
(2r - 3) d' - (r - 2)^2 (g' - d' + n) - 2r^2 + 3r - 9 &\geq 0 \label{e} \\
g - g' - n + 1 &\geq 0 \label{f} \\
r(d - d') - (r - 1)(g - g') + (r - 1) n - r^2 + 1 &\geq 0 \label{g} \\
n - 1 &\geq 0 \label{h} \\
d' - n &\geq 0 \label{i} \\
r (d - d') - (r - 4)(g - g') - 2n - 2r + 2 &\geq 0 \label{j} \\
2n + d + g' - d' - g - r - 2 &\geq 0; \label{k}
\end{align}
let $n$ be the minimal such integer.
Then any curve $f \colon C \cup_\Gamma D \to \pp^r$
of degree $d$ and genus~$g$,
so that $f|_C$ is a general BN-curve of degree $d'$
and genus $g'$; and $f|_D$ factors
as $\iota \circ f_D$, for $\iota \colon H \hookrightarrow \pp^r$ the inclusion
of a hyperplane $H \subset \pp^r$,
and $f_D$ a general BN-curve; and
such that $f(\Gamma)$ is a general set of $n$ points in $H$,
is an interior BN-curve.
\end{thm}

\begin{rem} If $(2r - 3) (d' + 1) - (r - 2)^2 g' - 2r^2 + 3r - 9 \geq 0$,
then Theorem~1.5 of~\cite{ibe} implies the hyperplane section of $f|_C$
is general. If $r \geq 4$, this implies the general such reducible
curve is an immersion.
So we get a curve in the boundary of the component of the Hilbert
scheme corresponding to BN-curves, as opposed to just for
the Kontsevich space.
\end{rem}

In the course of proving Theorem~\ref{main}, we also establish the following
slight variant
(which yields the same conclusion subject to a slightly different system
of inequalities):

\begin{thm} \label{main-p}
Let $d$, $g$, $d'$, $g'$, and $r$ be integers
which satisfy:
\begin{align}
g' &\geq 0 \tag{\ref*{b}$'$} \label{bp} \\
(r + 1) d - rg - r^2 - r &\geq 0 \tag{\ref*{c}$'$} \label{cp} \\
(r + 1) d' - rg' - r^2 - r &\geq 0 \tag{\ref*{d}$'$} \label{dp} \\
\intertext{Suppose there exists an integer $n$ satisfying:}
(2r - 3) (d' + 1) - (r - 2)^2 (g' - d' + n) - 2r^2 + 3r - 9 &\geq 0 \tag{\ref*{e}$'$} \label{ep} \\
g - g' - n &\geq 0 \tag{\ref*{f}$'$} \label{fp} \\
r(d - d') - (r - 1)(g - g') + (r - 1) n - r^2 &\geq 0 \tag{\ref*{g}$'$} \label{gp} \\
n - 1 &\geq 0 \tag{\ref*{h}$'$} \label{hp} \\
d' - n &\geq 0 \tag{\ref*{i}$'$} \label{ip} \\
r (d - d') - (r - 4)(g - g') - 2n - 2r - 2 &\geq 0 \tag{\ref*{j}$'$} \label{jp} \\
2n + d + g' - d' - g - r - 2 &\geq 0; \tag{\ref*{k}$'$} \label{kp}
\end{align}
let $n$ be the minimal such integer.
Then any curve $f \colon C \cup_\Gamma D \to \pp^r$
of degree $d$ and genus~$g$,
so that $f|_C$ is a general BN-curve of degree $d'$
and genus $g'$; and $f|_D$ factors
as $\iota \circ f_D$, for $\iota \colon H \hookrightarrow \pp^r$ the inclusion
of a hyperplane $H \subset \pp^r$,
and $f_D$ a general BN-curve; and
such that $f(\Gamma)$ is a general set of $n$ points in $H$,
is an interior BN-curve.
\end{thm}

\medskip

Our second goal is to prove the following theorem, which
sharpens Theorem~1.4 of~\cite{rbn} in the case $r = 3$:

\begin{thm} \label{main-3}
Let $f_i \colon C_i \to \pp^3$ (for $i \in \{1, 2\}$) be BN-curves
of degree $d_i$ and genus $g_i$, which pass through
a set $\Gamma \subset \pp^3$ of $n \geq 1$ general points.

Then $C_1 \cup_\Gamma C_2 \to \pp^3$ is a BN-curve, provided it has nonnegative Brill--Noether number,
unless $n = 2 d_1 = 2 d_2$.

(Since $C_1 \cup_\Gamma C_2 \to \pp^3$ is of degree $d_1 + d_2$ and genus $g_1 + g_2 + n - 1$,
the condition of having
nonnegative Brill--Noether number is just $4(d_1 + d_2) - 3(g_1 + g_2 + n - 1) - 12 \geq 0$.)

Furthermore, if both $f_i$ are general in some component of the space of BN-curves
passing through $\Gamma$, then $C_1 \cup_\Gamma C_2 \to \pp^3$ is an interior BN-curve.
\end{thm}

Several cases of Theorems~\ref{main} and~\ref{main-p} are already known: The case
$n \leq r + 2$ follows from Theorem~1.9 of~\cite{rbn};
the cases $r = 1$ and $r = 2$ follow from classical results
on the irreducibility of the Hurwitz space (c.f.\ \cite{cle})
and of the Severi variety (c.f.\ \cite{severi}).
\emph{We will therefore assume for the proof of Theorems~\ref{main} and~\ref{main-p} that:}
\begin{align}
r &\geq 3 \label{r3} \\
n &\geq r + 3. \label{nr3}
\end{align}

Since $\Gamma$ is a general set of points,
we may deform the curve $f$ appearing in Theorems~\ref{main} and~\ref{main-p}
to assume that $(f_D, \Gamma)$ is general in the component of $M_{g'', n}(H, d'')$
corresponding to BN-curves,
and that $f|_C$ is general in the component of
$M_{g'}(\pp^r, d')$
corresponding to BN-curves (hence is transverse to $H$).
Similarly, we may deform the curves $f_i$
appearing in Theorem~\ref{main-3}
to assume that the $(f_i, \Gamma)$ are both general
in the component of $M_{g_i, n}(\pp^3, d_i)$
corresponding to BN-curves.
In particular, by \eqref{r3}, we have
that $f$ is unramified in Theorems~\ref{main} and~\ref{main-p},
and that both $f_i$ are unramified in Theorem~\ref{main-3}.

We shall prove Theorems~\ref{main} and~\ref{main-p}
by simultaneous induction on $n$,
with Theorem~1.9 of~\cite{rbn}
(which implies both theorems when $n \leq	r + 2$)
serving as the base case.
Namely, we show first, in Section~\ref{main-to-main-p},
that Theorem~\ref{main} for any given value of $n$
implies Theorem~\ref{main-p} for the same value of $n$;
then, in Section~\ref{main-p-to-main},
we show that Theorems~\ref{main} and~\ref{main-p}
for any given value of $n$ imply Theorem~\ref{main}
for $n + 1$.
For Theorem~\ref{main-3}, our argument will also be by induction on $n$.

All of these inductive arguments will
use the strategy developed in \cite{rbn}
for showing certain reducible curves $C' \cup_\Gamma C'' \to \pp^r$
are BN-curves.
An overview of this strategy is given in
Section~3, part~II of the research announcement \cite{over}
for a series of papers --- including the present paper and \cite{rbn} --- which
builds up to a proof of the Maximal Rank
Conjecture. (This section of the research announcement may
be read independently from the remainder.)

\paragraph{Note:} Throughout this paper, we work over an algebraically
closed field of characteristic zero.

\subsection*{Acknowledgements}

The author would like to thank Joe Harris for
his guidance throughout this research,
as well as other members of the Harvard and MIT mathematics departments
for helpful conversations.
The author would also like
to acknowledge the generous
support both of the Fannie and John Hertz Foundation,
and of the Department of Defense
(NDSEG fellowship).

\section{Theorem~\ref{main} implies Theorem~\ref{main-p} \label{main-to-main-p}}

In this section, we show that Theorem~\ref{main} for a given value of $n$
implies Theorem~\ref{main-p} for the same value of $n$.

From \eqref{fp}, we have $g'' - 1 \geq 0$;
and from \eqref{gp}, we have
$\rho(d'' - 1, g'' - 1, r - 1) \geq 0$.
We may therefore (using Theorem~1.6 of~\cite{rbn}) specialize $f_D$ to a map from a reducible curve
$f_D^\circ \colon D' \cup_{\{p, q\}} \pp^1$,
with $D'$ of genus $g'' - 1$ and
$f_D^\circ|_{D'}$ of degree $d'' - 1$;
and $f_D^\circ|_{\pp^1}$ of degree~$1$;
and $\{p, q\}$ a set of two general points on $D'$.
By \eqref{jp} and Theorem~1.2 of~\cite{ibe},
$f_D^\circ|_{D'}$ can pass through $n$ general points;
in particular, we may specialize so that $\Gamma = \Gamma_1 \cup \Gamma_2$
is a set of $n$ general points consisting of
a set $\Gamma_1$ of $2$ points on $f_D^\circ(\pp^1)$,
and a set $\Gamma_2$ of $n - 2$ points on $f_D^\circ(\pp^1)$.
Note that since $\Gamma_1 \cup \{p, q\}$ is general,
$\Gamma_1$ and $\Gamma_2$ are independently general.

As in Lemma~3.6 of~\cite{rbn},
it suffices to show the resulting curve
$f^\circ \colon C \cup_{\Gamma_1 \cup \Gamma_2} (D' \cup_{\{p, q\}} \pp^1) \to \pp^r$
is a BN-curve and
$H^1(N_{f_D^\circ}) = H^1((f_D^\circ)^* \oo_{\pp^r}(1)(\Gamma_1 + \Gamma_2)) = 0$.
The vanishing of $H^1(N_{f_D^\circ})$ follows by combining
Lemmas~3.2, 3.3, and~3.4 of~\cite{rbn}.
Moreover the exact sequence
\[0 \to f_D^\circ|_{\pp^1}^* \oo_{\pp^r}(1)(\Gamma_1 - p - q) \simeq \oo_{\pp^1}(1) \to (f_D^\circ)^* \oo_{\pp^r}(1)(\Gamma_1 + \Gamma_2) \to f_D^\circ|_{D'}^* \oo_{\pp^r}(1)(\Gamma_2) \to 0\]
reduces the vanishing of $H^1((f_D^\circ)^* \oo_{\pp^r}(1)(\Gamma_1 + \Gamma_2))$
to the vanishing of
$H^1(f_D^\circ|_{D'}^* \oo_{\pp^r}(1)(\Gamma_2))$;
this in turn from
Lemma~6.2 of~\cite{rbn}, together with \eqref{kp}
which becomes upon rearrangement
(using \eqref{r3}):
\begin{equation} \label{for62}
(d'' - 1) - (g'' - 1) + (n - 2) \geq \max(2, r - 1) = r - 1.
\end{equation}

It thus remains to show $f^\circ$ is a BN-curve.
For this, we write $f^\circ$ as
\[f^\circ \colon  (C \cup_{\Gamma_1} \pp^1) \cup_{\Gamma_2 \cup \{p, q\}} D' \to \pp^r.\]

Note that each inequality ($k'$) for $(d, g, d', g', n)$ implies the
corresponding inequality
($k$) for $(d, g, d' + 1, g' + 1, n)$. Moreover,
each inequality $(k)$ for $(d, g, d' + 1, g' + 1, n - 1)$ implies the inequality
($k'$) for $(d, g, d', g', n - 1)$,
except for $k \in \{\ref{b}, \ref{d}, \ref{i}\}$
when ($k'$) for $(d, g, d', g', n - 1)$
follows from ($k'$) for $(d, g, d', g', n)$.
Thus, $(d, g, d' + 1, g' + 1, n)$ satisfies the inequalities of Theorem~\ref{main},
and $n$ is minimal with that property.

Note that $f^\circ|_{C \cup_{\Gamma_1} \pp^1}$
is a BN-curve by Theorem~1.6 of \cite{rbn}.
Showing that $f^\circ$ is a BN-curve
thus follows from
Theorem~\ref{main} (with the same value of $n$), since
$f^\circ|_{C \cup_{\Gamma_1} \pp^1}$ admits a deformation
still passing through $\Gamma_2 \cup \{p, q\}$
which is transverse to $H$ along $\Gamma_2 \cup \{p, q\}$
by Lemma~6.2 of \cite{rbn} together with \eqref{for62}.

\section{Proof of Theorem~\ref{main} \label{main-p-to-main}}

In this section, we show that Theorems~\ref{main} and~\ref{main-p} for $n - 1$
imply Theorem~\ref{main} for $n$. Together with the inductive
argument in the previous section, this will complete
the proofs of both Theorems~\ref{main} and~\ref{main-p}.

Since by assumption, $n$ is minimal subject to the system of inequalities
in Theorem~\ref{main}, one of these inequalities must cease to hold
when $n$ is replaced by $n - 1$.
Note that all inequalities except for
\eqref{g}, \eqref{h}, and \eqref{k} are nonincreasing in $n$,
and that \eqref{h}
continues to hold
when $n$ is replaced by $n - 1$
by \eqref{nr3}.
We must therefore be in one of two cases:

\paragraph{\boldmath Case 1: \eqref{k} ceases to hold when $n$ is replaced by $n - 1$:}
In other words, we have
\[2(n - 1) + d + g' - d' - g - r - 2 \leq -1.\]
Subtracting $r$ times this inequality from \eqref{g}, we obtain
upon rearrangement
\begin{equation} \label{btf}
g - g' - n + 1 \geq r(n - 3).
\end{equation}
In particular, combining this with \eqref{nr3},
the genus $g'' = g - g' - n + 1$ satisfies
\begin{equation} \label{btff}
g - g' - n + 1 \geq r.
\end{equation}
As \eqref{gp} implies
$\rho(d'' - r + 1, g'' - r, r - 1) \geq 0$,
we may therefore (using Theorem~1.6 of~\cite{rbn}) specialize $f_D$ to a map from a reducible curve
$f_D^\circ \colon D' \cup_\Delta \pp^1$,
with $D'$ of genus $g'' - r$ and
$f_D^\circ|_{D'}$ of degree $d'' - r + 1$;
and $f_D^\circ|_{\pp^1}$ of degree~$r - 1$;
and $\Delta$ a set of $r + 1$ points.

By Theorem~1.2 of~\cite{ibe},
$f_D^\circ|_{D'}$ can pass through $n - 1$ general points
provided that
\[(r - 2)(n - 1) \leq r (d'' - r + 1) - (r - 4)(g'' - r - 1) - 2(r - 1);\]
or substituting in $d'' = d - d'$ and $g'' = g - g' - n + 1$
and rearranging, provided that
\begin{equation} \label{btj}
r(d - d') - (r - 4)(g - g') - 2n - 4r \geq 0,
\end{equation}
which follows by adding
$\eqref{g} + 3 \cdot \eqref{btf} + (2r + 2) \cdot \eqref{nr3}$
to $3r^2 - 5r + 2 \geq 0$.

Note that $f_D^\circ|_{D'}$
can always pass through $r + 1$
general points in $H$, since $r + 1 \leq n - 1$ by \eqref{nr3},
and that $f_D^\circ|_{\pp^1}$ can pass through $r + 2$
general points in $H$ by Corollary~1.4 of \cite{aly}.
We may therefore degenerate so that
$f_D^\circ$
still passes through a set $\Gamma = \Gamma' \cup \{p\}$
of general points in $H$,
with $\# \Gamma' = n - 1 > 0$,
such that $f_D^\circ|_{D'}$ passes through $\Gamma'$
and $f_D^\circ|_{\pp^1}$ passes through $p$,
and such that $\Delta \cup \{p\}$
is a general set of $r + 2$ points in $H$.

As in Lemma~3.6 of~\cite{rbn},
it suffices to show the resulting curve
$f^\circ \colon C \cup_{\Gamma' \cup \{p\}} (D' \cup_\Delta \pp^1) \to \pp^r$
is a BN-curve and
$H^1(N_{f_D^\circ}) = H^1((f_D^\circ)^* \oo_{\pp^r}(1)(\Gamma' + p)) = 0$.
The vanishing of $H^1(N_{f_D^\circ})$ follows by combining
Lemmas~3.2, 3.3, and~3.4 of~\cite{rbn}.
Moreover the exact sequence
\[0 \to f_D^\circ|_{\pp^1}^* \oo_{\pp^r}(1)(p - \Delta) \simeq \oo_{\pp^1}(-1) \to (f_D^\circ)^* \oo_{\pp^r}(1)(\Gamma' + p) \to f_D^\circ|_{D'}^* \oo_{\pp^r}(1)(\Gamma') \to 0\]
reduces the vanishing of $H^1((f_D^\circ)^* \oo_{\pp^r}(1)(\Gamma' + p))$
to the vanishing of
$H^1(f_D^\circ|_{D'}^* \oo_{\pp^r}(1)(\Gamma'))$;
this follows in turn from Lemma~6.2 of \cite{rbn},
together with \eqref{k} which becomes upon rearrangement
\begin{equation} \label{for62ii}
(d'' - r + 1) - (g'' - r) + (n - 1) \geq \max(r + 1, r - 1) = r + 1.
\end{equation}

It thus remains to show $f^\circ$ is a BN-curve.
For this, we write $f^\circ$ as
\[f^\circ \colon  (C \cup_{\Gamma'} D') \cup_{\Delta \cup \{p\}} \pp^1 \to \pp^r.\]

Next, note that each inequality ($k$) for $(d, g, d', g', n)$ implies the
same inequality
($k$) for $(d - r + 1, g - r - 1, d', g', n - 1)$ --- except
for $k = \ref{f}$ when \eqref{f} for $(d - r + 1, g - r - 1, d', g', n - 1)$ follows from \eqref{btff},
for $k = \ref{h}$ when \eqref{h} for $(d - r + 1, g - r - 1, d', g', n - 1)$ follows from \eqref{nr3}, and
for $k = \ref{j}$ when \eqref{j} for $(d - r + 1, g - r - 1, d', g', n - 1)$ follows from \eqref{btj}.
Moreover,
each inequality $(k)$ for $(d - r + 1, g - r - 1, d', g', n - 2)$ implies the
same inequality
($k$) for $(d, g, d', g', n - 1)$,
except for $k \in \{\ref{c}, \ref{e}, \ref{i}\}$
when $(k)$ for $(d, g, d', g', n - 1)$ follows from $(k)$ for $(d, g, d', g', n)$).
Thus, $(d - r + 1, g - r - 1, d', g', n - 1)$
satisfies the inequalities of Theorem~\ref{main},
and $n$ is minimal with that property.

Consequently,
$f^\circ|_{C \cup_{\Gamma'} D'}$ is a BN-curve
by our inductive hypothesis for Theorem~\ref{main}.
Showing that $f^\circ$ is a BN-curve
thus follows from Theorem~1.9 of \cite{rbn}, since
$f^\circ|_{C \cup_{\Gamma'} D'}$ admits a deformation
still passing through $\Delta \cup \{p\}$
which is transverse to $H$ along $\Gamma' \cup \{p\}$
by Lemma~6.2 of \cite{rbn} together with \eqref{for62ii}.

\paragraph{\boldmath Case 2: \eqref{k} continues to hold when $n$ is replaced by $n - 1$, but \eqref{g} ceases to hold:}
Since \eqref{g} ceases to hold, we have
\[r(d - d') - (r - 1)(g - g') + (r - 1) (n - 1) - r^2 + 1 \leq -1.\]
Subtracting $(r + 1)$ times this equation from $r \cdot \eqref{c} + \eqref{f}$
and adding $r + 2 \geq 0$,
we obtain upon rearrangement
\begin{equation} \label{otd}
\rho(d' - 1, g', r) = (r + 1) (d' - 1) - r g' - r (r + 1) \geq 0.
\end{equation}

We may therefore (using Theorem~1.6 of~\cite{rbn}) specialize $f|_C$ to a map from a reducible curve
$f|_C^\circ \colon C' \cup_{\{p\}} \pp^1$,
with $C'$ of genus $g'$ and
$f|_C^\circ|_{C'}$ of degree $d' - 1$;
and $f|_C^\circ|_{\pp^1}$ factoring through $H$
of degree~$1$.

By Theorem~1.5 of~\cite{ibe} together with our assumption \eqref{e},
the hyperplane section of
$f|_C^\circ|_{C'}$ contains $n - 1$ general points;
and by inspection, $f|_C^\circ|_{\pp^1}$
passes through $2$ general points in $H$.
We may therefore degenerate so that
$f|_C^\circ$
still passes through a set $\Gamma = \Gamma' \cup \{q_1, q_2\}$
of general points in $H$,
with $\# \Gamma' = n - 2 > 0$ (c.f.\ \eqref{nr3}),
such that $f|_C^\circ|_{C'}$ passes through $\Gamma'$
and $f|_C^\circ|_{\pp^1}$ passes through $\{q_1, q_2\}$,
and such that $\Gamma' \cup \{f|_C^\circ(p)\}$
is a general set of $n - 1$ points in $H$.

As in Lemma~3.8 of~\cite{rbn},
it suffices to show 
$f^\circ \colon (C' \cup_{\{p\}} \pp^1) \cup_{\Gamma' \cup \{q_1, q_2\}} D \to \pp^r$
is a BN-curve,
$H^1(N_{f|_C^\circ}(-\Gamma' - q_1 - q_2)) = 0$, and
$d'' \geq g'' + r - 1 - (n - 2)$ (which upon rearrangement is exactly \eqref{k}).
The vanishing of $H^1(N_{f|_C^\circ}(-\Gamma' - q_1 - q_2))$ follows by combining
Lemmas~3.2, 3.3, and~3.4 of~\cite{rbn}.

It thus remains to show $f^\circ$ is a BN-curve.
For this, we write $f^\circ$ as
\[f^\circ \colon  C' \cup_{\Gamma' \cup \{p\}} (D \cup_{\{q_1, q_2\}} \pp^1) \to \pp^r.\]

Next, note that each inequality ($k$) for $(d, g, d', g', n)$ implies the
corresponding inequality
($k'$) for $(d, g, d' - 1, g', n - 1)$ --- except
for $k = \ref{d}$ when \eqref{dp} for $(d, g, d' - 1, g', n - 1)$ follows from \eqref{otd},
for $k = \ref{h}$ when \eqref{hp} for $(d, g, d' - 1, g', n - 1)$ follows from \eqref{nr3},
and for $k = \ref{k}$ when \eqref{kp} for $(d, g, d' - 1, g', n - 1)$ follows from \eqref{k} for $(d, g, d', g', n - 1)$.
Moreover,
each inequality $(k')$ for $(d, g, d' - 1, g', n - 2)$ implies the
corresponding inequality
($k$) for $(d, g, d', g', n - 1)$,
except for $k = \ref{j}$
when $\eqref{j}$ for $(d, g, d', g', n - 1)$ follows from $\eqref{j}$
for $(d, g, d', g', n)$.
Thus, $(d, g, d' - 1, g', n - 1)$
satisfies the inequalities of Theorem~\ref{main-p},
and $n$ is minimal with that property.

By Theorem~1.6 of~\cite{rbn},
$D \cup_{\{q_1, q_2\}} \pp^1 \to H$ is a BN-curve.
Our inductive hypothesis for Theorem~\ref{main-p}
thus shows $f^\circ$ is a BN-curve as desired.

\section{Proof of Theorem~\ref{main-3}}

To prove Theorem~\ref{main-3}, we will argue by induction on $n$.
Write $\rho_i = 4d_i - 3g_i - 12$ for the Brill--Noether number of $f_i$.
Note that, since $f_i$ passes through $n$ general points, we have from Theorem~\ref{thm11}
that $n \leq 2d_i$; by assumption one of these inequalities is strict.
Note also that by assumption, $4(d_1 + d_2) - 3(g_1 + g_2 + n - 1) - 12 \geq 0$;
upon rearrangement, this becomes
\begin{equation} \label{n-upper}
n \leq \frac{\rho_1 + \rho_2 + 15}{3}.
\end{equation}
We will separately consider several cases:

\begin{proof}[Proof of Theorem~\ref{main-3} when $\rho_1 \geq 4$ and $n \leq 2 d_1 - 1$.]
If $\rho_2 \geq 4$ and $n \leq 2 d_2 - 1$, then
by permuting indices if necessary, we may assume without loss of generality
that $d_1 \geq d_2$, and that $g_1 \leq g_2$ if $d_1 = d_2$.
On the other hand, if $n \geq 2 d_2$, then $2 d_1 - 1 \geq n \geq 2 d_2$,
and so $d_1 > d_2$.
We may therefore assume that
$d_1 \geq d_2$, and $g_1 \leq g_2$ if $d_1 = d_2$,
subject only to the assumption that $\rho_2 \geq 4$.

Since $\rho(d, g, 3)$ is increasing in $d$ and decreasing in $g$, we conclude that
$\rho_2 \geq 4$ implies
$\rho_2 \leq \max(\rho(d_1, g_1, 3), \rho(d_1 - 1, 0, 3)) = \max(\rho_1, 4d_1 - 16)$.
On the other hand, $\rho_2 \leq 3$ implies $\rho_2 \leq 3 \leq 4 \leq \rho_1$.
We may therefore assume in all cases that
\[\rho_2 \leq \max(\rho_1, 4d_1 - 16).\]
Combining this with \eqref{n-upper}, we obtain
\[n \leq \frac{\rho_1 + \rho_2 + 15}{3} \leq \frac{\rho_1 + \max(\rho_1, 4d_1 - 16) + 15}{3} = \frac{\max(8 d_1 - 6 g_1 - 9, 8 d_1 - 3 g_1 - 13)}{3}.\]
In particular, if $(d_1, g_1) \in \{(6, 2), (7, 4)\}$, then $n \leq 10$.
Thus,
\[n - 1 \leq \begin{cases}
2(d_1 - 1) & \text{if $(d_1 - 1, g_1) \notin \{(5, 2), (6, 4)\}$;} \\
9 & \text{if $(d_1 - 1, g_1) \in \{(5, 2), (6, 4)\}$.}
\end{cases}\]

Since $\rho(d_1 - 1, g_1, 3) = \rho_1 - 4 \geq 0$, we may (using Theorem~1.6 of~\cite{rbn})
specialize $f_1$ to a map from a reducible curve $f_1^\circ \colon C_1' \cup_p \pp^1 \to \pp^3$,
with $C_1'$ of genus $g_1$, and $f_1^\circ|_{C_1'}$ of degree $d_1 - 1$, and $f_1^\circ|_{\pp^1}$ of degree $1$.
By the above inequality, we may do this so
$f_1^\circ$ still passes through a set $\Gamma = \Gamma' \cup \{x, y\}$
of $n$ general points, such that $f_1^\circ|_{C_1'}$ passes through $\Gamma'$,
and $f_1^\circ|_{\pp^1}$ passes through $\{x, y\}$, and such that $\Gamma' \cup \{p\}$
is a general set of $n - 1$ points.

As in Lemma~3.5 of \cite{rbn},
it suffices to show 
$(C_1' \cup_{\{p\}} \pp^1) \cup_{\Gamma' \cup \{x, y\}} C_2 \to \pp^3$
is a BN-curve. For this, we simply rewrite this map as
$C_1' \cup_{\Gamma' \cup \{p\}} (\pp^1 \cup_{\{x, y\}} C_2) \to \pp^3$,
which is a BN-curve
by Theorem~1.6 of~\cite{rbn} and our inductive hypothesis.
\end{proof}

\begin{proof}[Proof of Theorem~\ref{main-3} when $\rho_1 \geq 4$ and $n = 2 d_1$.]
From \eqref{n-upper}, we obtain
\[\rho_2 \geq 3n - 15 - \rho_1 = \frac{\rho_1}{2} + 3 + \frac{9}{2} g_1 \geq \frac{4}{2} + 3 = 5 \geq 4.\]
And since by assumption we do not
have $n = 2 d_1 = 2 d_2$, we have $n \leq 2 d_2 - 1$.
Exchanging indices, we are thus in the previous case.
\end{proof}

This completes the proof when $\rho_1 \geq 4$, and thus by symmetry when $\rho_2 \geq 4$.
Exchanging indices if necessary, it therefore remains to consider the case $\rho_1 \leq \rho_2 \leq 3$.

\begin{proof}[Proof of Theorem~\ref{main-3} when $\rho_1 \leq \rho_2 \leq 3$.]
In this case, we argue by induction on $\rho_1$.
If $\rho_1 = 0$, then using \eqref{n-upper},
the result follows from Theorem~1.6 of~\cite{rbn}.

For the inductive step, we therefore suppose $1 \leq \rho_1 \leq 3$ (which forces $d_1 \geq 4$ and $g_1 \geq 1$).
Note that \eqref{n-upper} gives
\[n \leq \frac{\rho_1 + \rho_2 + 15}{3} \leq \frac{3 + 3 + 15}{3} = 7,\]
with equality only if $\rho_1 = 3$ (which forces $d_1 \geq 6$).
In particular,
\[n \leq \begin{cases}
2(d_1 - 1) & \text{if $(d_1 - 1, g_1 - 1) \notin \{(5, 2), (6, 4)\}$;} \\
9 & \text{if $(d_1 - 1, g_1 - 1) \in \{(5, 2), (6, 4)\}$.}
\end{cases}\]

Since $\rho(d_1 - 1, g_1 - 1, 3) = \rho_1 - 1 \geq 0$, we may (using Theorem~1.6 of~\cite{rbn})
specialize $f_1$ to a map from a reducible curve $f_1^\circ \colon C_1' \cup_{\{p, q\}} \pp^1 \to \pp^3$,
with $C_1'$ of genus $g_1 - 1$, and $f_1^\circ|_{C_1'}$ of degree $d_1 - 1$, and $f_1^\circ|_{\pp^1}$ of degree $1$.
By the above inequality, we may do this so
$f_1^\circ$ still passes through a set $\Gamma = \Gamma' \cup \{x, y\}$
of $n$ general points, such that $f_1^\circ|_{C_1'}$ passes through $\Gamma'$,
and $f_1^\circ|_{\pp^1}$ passes through $\{x, y\}$, and such that $\Gamma' \cup \{p, q\}$
is a general set of $n$ points.

As in Lemma~3.5 of \cite{rbn},
it suffices to show 
$(C_1' \cup_{\{p, q\}} \pp^1) \cup_{\Gamma' \cup \{x, y\}} C_2 \to \pp^3$
is a BN-curve. For this, we simply rewrite this map as
$C_1' \cup_{\Gamma' \cup \{p, q\}} (\pp^1 \cup_{\{x, y\}} C_2) \to \pp^3$,
which is a BN-curve
by Theorem~1.6 of~\cite{rbn}, together with either an application of our inductive hypothesis
or one of the two previously-considered cases.
\end{proof}
\bibliographystyle{amsplain.bst}
\bibliography{mrcbib}

\end{document}